\documentclass[10pt]{article}
\usepackage{amssymb,amsthm,amsmath}

\usepackage{graphics}
\usepackage[all,color]{xy}

\newtheorem{theorem}{Theorem}[section]

\newtheorem{proposition}[theorem]{Proposition}

\newtheorem{lemma}[theorem]{Lemma}

\newtheorem*{thm::tukeymaxideals}{Theorem~\ref{tukeymaxideals}}
\newtheorem*{thm::fsigma}{Theorem~\ref{fsigma}}
\newtheorem*{thm::omegaturkey}{Theorem~\ref{omegaturkey}}
\newtheorem*{thm::completeP}{Theorem~\ref{completeP}}
\newtheorem*{thm::completeLCP}{Theorem~\ref{completeLCP}}

\theoremstyle{definition}

\newtheorem{definition}[theorem]{Definition}

\theoremstyle{remark}

\newtheorem{claim}[theorem]{Claim}

\newtheorem{question}[theorem]{Question}
\newtheorem*{remark*}{Remark}
\newtheorem*{remarks*}{Remarks}

\def\cZ{\mathcal{Z}}
\def\restrict{\upharpoonright}
\newcommand{\cP}{\mathcal{P}}
\newcommand{\NWD}{\mathrm{NWD}}

\begin{document}

\title{Maximal Tukey types, P-ideals and the weak Rudin-Keisler order}
\author{Konstantinos A. Beros\\
{\small Miami University}
\and 
Paul B. Larson%
\thanks{The research of 
the second author is supported by NSF grants DMS-1201494 and DMS-1764320. Acknowledgment: this preprint has not undergone peer review or any post-submission improvements or corrections. The Version of Record of this article is published in the Archive for Mathematical Logic, and is available online at https://doi.org/10.1007/s00153-023-00897-z.}\\
{\small Miami University}}
\date{}

\maketitle

\begin{abstract}
In this paper, we study some new examples of ideals on $\omega$ with maximal Tukey type (that is, maximal among partial orders of size continuum).  This discussion segues into an examination of a refinement of the Tukey order -- known as the {\em weak Rudin-Keisler order} -- and its structure when restricted to these ideals of maximal Tukey type.  Mirroring a result of Fremlin \cite{fremlin.tukey} on the Tukey order, we also show that there is an analytic P-ideal above all other analytic P-ideals in the weak Rudin-Keisler.
\end{abstract}

\section{Introduction}

A fundamental way of studying partial orders is by examining their cofinal structure, i.e., the structure of their unbounded subsets.  In 1940, John Tukey \cite{tukey} introduced one of the simplest methods of comparing the cofinal structure of different partial orders.

\begin{definition}
If $P$ and $Q$ are partial orders, $P$ is {\em Tukey-reducible} to $Q$ (written $P \leq_{\rm Tukey} Q$) iff there is a map $f : P \rightarrow Q$ such that the $f$-image of an unbounded set in $P$ is unbounded in $Q$.  Two partial orders have the same {\em Tukey type} (or {\em cofinal type}) iff each is Tukey-reducible to the other.
\end{definition}

Tukey himself examined cofinal types in the context of convergence in topological spaces.  It was John Isbell \cite{isbell.cofinal_types} who brought the Tukey order into the realm of combinatorial set theory when he showed that the ideal $\mathcal Z_0$ of asymptotic density zero reals is not Tukey reducible to $\omega^\omega$ (equipped with the eventual domination order).  Subsequent to Isbell's work, Fremlin \cite{fremlin.tukey}, Louveau--Velickovic \cite{louveau-velickovic.ideals-cofinal-types}, Solecki--Todorcevic \cite{solecki-todorcevic.directed,solecki-todorcevic.avoiding} and others have contributed to understanding the Tukey types of those Borel ideals and partial orders which arise naturally in mathematics.  Restricting to some key ideals of interest, the Tukey order looks like
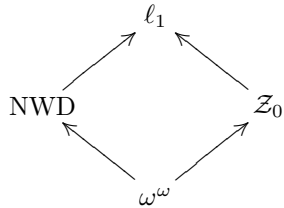
\[
\xy
(0,0)*++{\xy (-10,0)*+{}="init";
(0,0)*+[o]=<20pt>\hbox{${\rm NWD}$}*\frm{}="nwd";
(15,12)*+[o]=<20pt>\hbox{$\ell_1$}*\frm{}="ell";
(15,-12)*+[o]=<20pt>\hbox{$\omega^{\omega}$}*\frm{}="omega";
(30,0)*+[o]=<20pt>\hbox{$\mathcal Z_0$}*\frm{}="zee";
(45,-10)*+{}="lp-1-1";
(45,10)*+{}="lp-1-2";
(40,-15)*+{}="lp-2-1";
(35,-8)*+{}="lp-2-2";
"nwd";"ell"**\dir{-} ?>*\dir2{>};
"omega";"nwd"**\dir{-} ?>*\dir2{>};
"zee";"ell"**\dir{-} ?>*\dir2{>};
"omega";"zee"**\dir{-} ?>*\dir2{>};
\endxy}="transducer";
\endxy
\]
where NWD is the ideal of nowhere dense compact subsets of $2^\omega$ and $\ell_1$ is the ideal of sets $X \subseteq \omega$ such that $\sum_{n \in X} \frac{1}{n+1}$ is finite.  In this diagram, an arrow $P \longrightarrow Q$ indicates the strict Tukey reduction $P <_{\rm Tukey} Q$.

In an extensive paper from 1991, Fremlin \cite{fremlin.tukey} made major advances towards filling in the diagram above.  Of particular interest here, he showed that every Polishable ideal is Tukey-reducible to $\ell_1$.  (An ideal is {\em Polishable} if it admits a Polish topology consistent with its Borel structure and under which its algebraic operations are continuous.)  This particular result of Fremlin illustrates that the Tukey order ignores definable complexity to a certain degree: whereas Polishable ideals are in general $F_{\sigma \delta}$, the ideal $\ell_1$ is merely $F_\sigma$.  There are a variety of ways to strengthen the Tukey order in such a way that both cofinal structure and definable complexity are preserved.  For the purposes of this paper, the ``weak Rudin-Keisler order'' is the most suitable.  This variant on the standard Rudin-Keisler order was defined by the first author in \cite{beros.wrk}.

\begin{definition}
Given ideals $I \subseteq \mathcal P (A)$ and $J \subseteq \mathcal P (B)$, $I$ is {\em weak Rudin-Keisler reducible} to $J$ iff there is an infinite set $B' \subseteq B$ and a map $f: B' \rightarrow A$ such that $X \in I \iff f^{-1} [X] \in J$ for each $X \subseteq A$.  In this case, write $I \leq_{\rm wRK} J$ and call $f$ a {\em weak Rudin-Keisler reduction} (or {\em wRK-reduction}) of $I$ to $J$.
\end{definition}

\begin{remarks*}
For ideals $I$ and $J$,
\[
I \leq_{wRK} J \iff (\exists \mbox{ infinite } X) (I \leq_{\rm RK} J \cap \mathcal P (X))
\]
where $\leq_{RK}$ is the standard Rudin-Keisler order on ideals.  Also, notice that since preimages respect unions,
\[
I \leq_{\rm wRK} J \implies (I , \subseteq) \leq_{\rm Tukey} (J , \subseteq).
\]
Thus, the wRK order is a weakening of the Rudin-Keisler order and a strengthening of the Tukey order.  Furthermore, observe that, if $f: \omega \rightarrow \omega$, the map $X \mapsto f^{-1} [X]$ is continuous on $2^\omega$.  It follows that any wRK reduction is a Wadge reduction and hence preserves definable complexity as well.

It must be noted that the Tukey order also preserves definable complexity in certain cases.  Corollary 5.4 in Solecki--Todorcevic \cite{solecki-todorcevic.directed} implies an excellent example this:  if $I , J \subseteq 2^\omega$ are ideals with $I \leq_{\rm Tukey} J$ and $J$ is analytic, then $I$ is analytic as well.
\end{remarks*}

The present work is divided into two parts.  Sections \ref{tukeymax} and \ref{belownwd} below concern Tukey reductions and ideals of maximal Tukey type.  Sections \ref{wftrees}, \ref{lcpol} and \ref{pideals} below concern wRK-reductions and P-ideals.

Two ideals of interest in the first part of the paper are 
\[
I_{\rm wf} = \{ X \subseteq 2^{<\omega} : (\forall y) (\forall^\infty n) (y \upharpoonright n \notin X)\}
\]
and
\[
I_\omega = \{ X \subseteq 2^{<\omega} : X \mbox{ is a finite union of $\preceq$-antichains}\}
\]
where $\preceq$ is the extension order on binary strings.  More generally, for an additively closed ordinal $\alpha < \omega_1$, let 
\[
I_\alpha = \{ X \subseteq 2^{<\omega} : \mbox{there is an order-preserving } f : (X , \succ) \rightarrow (\alpha , <) \}.
\]
(Section \ref{prelims} below gives a more detailed definition of the $I_\alpha$ and explains the requirement that $\alpha$ be additively closed.  This requirement is related to combinatorial results of Ryan Causey \cite{causey.ellp}.)

From the standpoint of the Tukey order, the most important observation is that all of these ideals are of maximal Tukey type:

\begin{thm::tukeymaxideals}
Every partial order of cardinality continuum (or less) is Tukey-reducible to $I_{\rm wf}$ and all $I_\alpha$ (where $\alpha$ is additively closed).
\end{thm::tukeymaxideals}

Although the existence of ideals of maximal Tukey type is already known (see for instance Louveau--Velickovic \cite{louveau-velickovic.ideals-cofinal-types}), Theorem~\ref{tukeymaxideals} provides some new examples.  The next main result concerning Tukey types establishes a connection between topological and cofinality properties for ideals Tukey-reducible to NWD.

\begin{thm::fsigma}
If $I \subseteq \mathcal P (\omega)$ is an $F_\sigma$ ideal and $I \leq_{\rm Tukey} {\rm NWD}$, then $I$ is countably generated.
\end{thm::fsigma}

It is also worth noting that not all $F_\sigma$ ideals are countably generated.  For instance, the summable ideal $\ell_1$ is not countably generated.

Theorem~\ref{fsigma} is of interest since NWD itself is not countably generated.  Moreover, there are $F_{\sigma \delta}$ ideals on countable sets which are Tukey-below NWD, but not countably generated.  For instance, the ideal
\[
\emptyset \times {\rm Fin} = \{ X \subseteq \omega \times \omega : (\forall m) ( \{ n : (m,n) \in X \} \mbox{ is finite})\}
\]
is Tukey-reducible to NWD.  To see this, let $s_0 , s_1 , \ldots$ be an enumeration of $2^{<\omega}$ and observe that the map
\[
X \mapsto \Big\{ \underbrace{0 {}^\smallfrown \ldots {}^\smallfrown 0}_{m \mbox{ \footnotesize times}} {}^\smallfrown  1 {}^\smallfrown s_i {}^\smallfrown 0 {}^\smallfrown \ldots  : (\exists n) (i \leq n \ \& \  (m,n) \in X) \Big\} \cup \Big\{ 0 {}^\smallfrown 0 {}^\smallfrown \ldots \Big\}
\]
is a Tukey-reduction of $I_1$ to NWD.  On the other hand, $I_1$ is not countably generated since there is no countable dominating family in $\omega^\omega$.


Turning to the wRK order, the next main result of this paper is

\begin{thm::omegaturkey}
If $\alpha$ is a countable ordinal, $I_\omega \leq_{\rm wRK} I_{\omega^{\alpha + 1}}$.
\end{thm::omegaturkey}

As a counterpoint to this theorem, Proposition \ref{nowrk} below provides some non-reduction results between $\ell_1$, $I_{\rm wf}$ and the $I_\alpha$.  In short, all of these ideals are wRK-incomparable.

In the context of the wRK-order, the next two results are analogous to Fremlin's result \cite[Theorem 2B]{fremlin.tukey} that all Polishable ideals are Tukey-reducible to $\ell_1$, i.e., $\ell_1$ is a Tukey-complete Polishable ideal.

\begin{thm::completeLCP}
There exists a wRK-complete locally compact Polishable ideal, i.e., there exists a locally compact Polishable ideal $I_* \subseteq \mathcal P (\omega)$ such that $J \leq_{\rm wRK} I_*$ for each locally compact Polishable ideal $J$.
\end{thm::completeLCP}

Recall that a {\em P-ideal} is an ideal $I$ such that every countable increasing sequence in $I$ has a pseudo-union in $I$.

\begin{thm::completeP}
There exists a wRK-complete analytic P-ideal, i.e., there is an analytic P-ideal $I_{\rm max}$ such that $J \leq_{\rm wRK} I_{\rm max}$ for each analytic P-ideal $J$.
\end{thm::completeP}

In light of Solecki's characterization \cite{solecki.ideals} of analytic P-ideals as Polishable ideals, Theorem~\ref{completeP} is equivalent to the existence of a wRK-complete Polishable ideal.  Details of Solecki's work are discussed in Section \ref{pideals} as his results are crucial to the proof of Theorem~\ref{completeP}.

\section{Preliminaries}\label{prelims}


\subsection{Notation}

This paper uses standard notation for the most part.  For the reader's reference, this section contains some key pieces of notation.  Given a binary string $s$, let $|s|$ denote the length of $s$.  Let $s \preceq t$ indicate that the string $s$ is a prefix of $t$.  If $y$ is a finite or infinite string, let $y \upharpoonright n$ denote the length $n$ initial segment of $y$.  If $s$ and $t$ are strings with no common prefix other than the empty string, write $s \perp t$.

Let $2^{<\omega}$ denote the collection of all binary strings and $2^\omega$ the Cantor space, i.e., the space of all binary sequences with the product topology.

If $F$ is a finite set, let $|F|$ denote the cardinality of $F$.  (This represents a notational overload with the notation $|s|$ above.  Nevertheless, the intended meaning will always be clear from context.)


\subsection{Well-founded trees}

As mentioned in the Introduction, the ideals of interest in the present work consist of sets of binary strings with no infinite $\preceq$-chains.

\begin{definition}
Given $X \subseteq 2^{<\omega}$ (not necessarily closed under taking prefixes), let
\[
[X] = \{ y \in 2^\omega : (\exists^\infty n) (y \upharpoonright n \in X)\}.
\]
As in the Introduction, let $I_{\rm wf}$ be the ideal
\[
\{ X \subseteq 2^{<\omega} : [X] = \emptyset\}.
\]
and note that $I_{\rm wf}$ is a complete $\mathbf \Pi^1_1$ subset of $\mathcal P (2^{<\omega})$.
\end{definition}

\begin{definition}
Suppose $X \subseteq 2^{<\omega}$ with $[X] = \emptyset$.
\begin{itemize}
	\item Define the {\em rank function} $\rho : X \rightarrow \omega_1$ as follows:
	\begin{itemize}
		\item if $u \in X$ and $t \notin X$ for all $t \succ u$, let $\rho (u) = 0$;
		\item otherwise, let $\rho (u) = \sup \{ \rho (t) + 1 : t \in X \mbox{ and } u \prec t\}$.
	\end{itemize}
	
	\item If $\rho : X \rightarrow \omega_1$ is the rank function defined above, define the {\em rank} of $X$ by 
	\[
	{\rm rank} (X , \preceq) = \sup \{ \rho (u) : u \in X \}.
	\]
	Note that ${\rm rank} (X , \preceq) < \omega_1$ since $X$ is countable and $[X] = \emptyset$.
\end{itemize}
\end{definition}

\begin{definition}
For each ordinal $\alpha < \omega_1$, let $I_\alpha$ be the collection of $X \subseteq 2^{< \omega}$ such that ${\rm rank} (X , \preceq) < \alpha$.
\end{definition}

In general, each $I_\alpha$ is closed under subsets, but may not be closed under finite unions, i.e., $I_\alpha$ is not necessarily an ideal.

\begin{proposition}\label{addclosed}
For an ordinal $\alpha < \omega_1$, the set $I_\alpha$ is an ideal iff $\alpha$ is additively closed, i.e., $\beta , \gamma < \alpha \implies \beta + \gamma < \alpha$.
\end{proposition}

The key to proving this proposition is a combinatorial lemma -- due to Ryan Causey -- which may be extracted by combining Proposition 2.3 (parts ii and iii) and Corollary 3.9 in Causey \cite{causey.ellp}.

\begin{lemma}[R.~Causey]\label{causey-lemma}
Suppose that $\alpha$ is an additively closed ordinal and $A \subseteq 2^{<\omega}$ with ${\rm rank} (A , \preceq) \geq \alpha$.  If $f : A \rightarrow 2$ is any function (i.e., a 2-coloring), there exists $B \subseteq A$ such that ${\rm rank} (B , \preceq) = \alpha$ and $f \upharpoonright B$ is constant.
\end{lemma}

This lemma is essentially a pigeonhole principle for well-founded trees.

\begin{proof}[Proof of Proposition \ref{addclosed}]
First of all, assume that $\alpha$ is additively closed.  Towards a contradiction, suppose that $I_\alpha$ is not an ideal.  Let $X , Y \in I_\alpha$ be such that $X \cup Y \notin I_\alpha$.  Since $X \cup Y$ is well-founded, it must be that ${\rm rank} (X \cup Y , \preceq) \geq \alpha$.  Define a 2-coloring on $X \cup Y$ by $f(u) = 0$ if $u \in X$ and $f(u) = 1$ otherwise.

By Lemma~\ref{causey-lemma} above, there is a set $B \subseteq X \cup Y$ such that ${\rm rank} (B , \preceq) = \alpha$ and $f \upharpoonright B$ is constant.  If $f \upharpoonright B \equiv 0$, then ${\rm rank} (X , \preceq) = \alpha$ and, if $f \upharpoonright B \equiv 1$, then ${\rm rank} (Y , \preceq) \geq \alpha$.  In either case, this would contradict the assumption that $X , Y \in I_\alpha$.

Now suppose that $\alpha$ is not additively closed with $\beta , \lambda < \alpha$ such that $\beta + \lambda \geq \alpha$.  Let $X , Y \subseteq 2^{<\omega}$ have no infinite $\preceq$-chains and be such that ${\rm rank} (X , \preceq) = \beta$ and ${\rm rank} (Y , \preceq) = \lambda$.  Let 
\[
\tilde X = \{ t {}^\smallfrown u : t \mbox{ is a terminal node in } Y \mbox{ and } u \in X\}.
\]
It follows that $\tilde X$ also has rank $\beta$.  Furthermore, since the terminals in $Y$ become roots of copies of $X$ in $\tilde X \cup Y$, the terminals of $Y$ have rank $\beta$ in $\tilde X \cup Y$.  Thus, by induction on the rank of nodes in $Y$,
\[
{\rm rank} (\tilde X \cup Y , \preceq) = \beta + \lambda \geq \alpha.
\]
Since $\tilde X , Y \in I_\alpha$, this shows that $I_\alpha$ is not an ideal and completes the proof.
\end{proof}

\section{Tukey-maximal ideals}\label{tukeymax}

The main result of this section (Theorem~\ref{tukeymaxideals}) shows that $I_{\rm wf}$ and the $I_\alpha$ are Tukey-maximal among partial orders of size continuum.

\begin{definition}
A partial order $P$ of size continuum has {\em maximal Tukey type} if every partial order of size continuum or less is Tukey-reducible to $P$.
\end{definition}

The key to proving that a partial order is of maximal Tukey type is establishing that it contains a large ``strongly unbounded'' subset.

\begin{definition}
Suppose $P$ is a partial order.  An infinite subset $X \subseteq P$ is {\em strongly unbounded} iff every infinite subset of $X$ is unbounded in $P$.
\end{definition}

Suppose that $P$ is a partial order which contains a strongly unbounded set $X$ of size continuum.  If $Q$ is any partial order of size continuum and $f : Q \rightarrow X$ is injective, then is a Tukey map from $Q$ into $P$.  In particular, $P$ has maximal Tukey type.  Conversely,

\begin{proposition}
If $P$ has maximal Tukey type, $P$ must contain an uncountable strongly unbounded set.
\end{proposition}

\begin{proof}
Suppose that $P$ is of maximal Tukey type, $Q$ is a partial order of size continuum with an uncountable strongly unbounded subset $X$, and $f : Q \rightarrow P$ is a Tukey map.  

{\em Claim.}  $f[X]$ is uncountable.

Otherwise, there exists $y \in f[X]$ such that $X \cap f^{-1} [ \{ y \}]$ is infinite and hence unbounded in $Q$ -- since $X$ is strongly unbounded.  In particular, $f$ maps an unbounded set to the singleton $\{ y \}$.  This would contradict the assumption that $f$ is a Tukey map.

{\em Claim.}  $f[X]$ is strongly unbounded.

Indeed, suppose that $Y \subseteq f[X]$ is infinite.  It follows that $f^{-1} [Y] \cap X$ is infinite and hence unbounded.  Thus, $Y$ is unbounded since $f$ is a Tukey map.
\end{proof}

The next proposition shows that $\ell_1$ is not of maximal Tukey type.  In particular, no analytic $P$-ideal is of maximal type.

\begin{proposition}
The ideal $\ell_1$ contains no uncountable strongly unbounded set.
\end{proposition}

\begin{proof}
Given $y \in \ell_1$, let $S(y)$ denote the sum $\sum_{n \in y} \frac{1}{n+1}$.  Suppose $X \subseteq \ell_1$ is uncountable.  The first objective is to obtain binary strings $s_k$ and uncountable sets $X_k$ such that the following hold:
\begin{itemize}
	\item $\langle \, \rangle = s_0 \prec s_1 \prec \ldots$,
	\item $X = X_0 \supseteq X_1 \supseteq \ldots$ and
	\item if $k \geq 1$ and $y \in X_k$, then $s_1 \prec y$ and $S(y \setminus |s_k|) \leq 2^{-k}$.
\end{itemize}
To accomplish this, suppose inductively that $s_0 \prec \ldots \prec s_k$ and $X_0 \supseteq \ldots \supseteq X_k$ are given as above.  For each $s \succ s_k$, let
\[
Z_s = \{ y \in X_k : s \prec y \mbox{ and } S(y \setminus |s|) < 2^{-k-1}\}.
\]
Since $X_k$ is uncountable and the $Z_s$ partition $X_k$, there must be $s \succ s_k$ such that $Z_s$ is uncountable and $X_k \setminus Z_s \neq \emptyset$.  Let $s_{k+1} = s$ and $Z_s = X_{k+1}$.

To complete the proof, choose $y_k \in X_k \setminus X_{k+1}$ for each $k \geq 1$.  Thus, since $s_k \prec y_k$ (for $k \in \omega$) and $s_0 \prec s_1 \prec \ldots$,
\[
\bigcup_{k \geq 1} y_k \subseteq \left( \bigcup_{k \geq 1} y_k \cap \Big[|s_{k-1}| , |s_k| \Big) \right) \cup \left( \bigcup_{k \geq 1} y_k \setminus |s_k |\right)
\]
By choice of $s_k$, it follows that $S\left( y_k \cap \big[|s_{k-1}| , |s_k| \big) \right) < 2^{-k+1}$ (for $k \geq 1$) and $S (y_k \setminus |s_k|) < 2^{-k}$.  Therefore $S ( \bigcup_{k \geq 1} y_k) \leq 3$ and hence $\{ y_k : k \in \omega\}$ is an infinite bounded subset of $X$.  In particular, $X$ is not strongly unbounded.
\end{proof}

Recall now the definitions of $I_{\rm wf}$ and the $I_\alpha$ (for additively closed $\alpha$) from Section~\ref{prelims}, as well as the notation $[X]$ (for $X \subseteq 2^{<\omega}$).

\begin{theorem}\label{tukeymaxideals}
Every partial order of cardinality continuum (or less) is Tukey-reducible to $I_{\rm wf}$ and each $I_\alpha$, i.e., $I_{\rm wf}$ and the $I_\alpha$ are of maximal Tukey type.
\end{theorem}

\begin{proof}
It suffices to show that there are strongly unbounded subsets of cardinality continuum in the ideals $I_{\rm wf}$ and $I_\alpha$.  In fact, there is a perfect subset of $2^{<\omega}$ which is strongly unbounded in both $I_{\rm wf}$ and the $I_\alpha$.  Specifically, given $y \in 2^\omega$, let 
\[
X_y = \{ s \in 2^{<\omega} : s \nprec y \mbox{, but } s \upharpoonright (|s| - 1) \prec y\}.
\]
As a $\preceq$-antichain, each $X_y$ is in $I_\omega$ and hence is $I_{\rm wf}$ and the other $I_\alpha$.  The map
\[
y \mapsto X_y
\]
is continuous and therefore 
\[
P = \{ X_y : y \in 2^\omega \}
\]
is perfect.  To see that $P$ is strongly unbounded, suppose that $y_0 , y_1 , \ldots  \in 2^\omega$ are distinct.  The goal is to show that $\{ X_{y_k} : k \in \omega\}$ is unbounded.  Passing to a subsequence, assume that $y_k \rightarrow y$ as $k \rightarrow \infty$ for some $y$ different from all the $y_k$.  For each $k$, let $n_k \in \omega$ be largest such that $y_k \upharpoonright n_k = y \upharpoonright n_k$.  It follows that $y \upharpoonright (n_k + 1) \in X_{y_k}$ for each $k \in \omega$.  In particular, 
\[
y \in \left[ \bigcup \{  X_{y_k} : k \in \omega\} \right].
\]
In other words, $\{ X_{y_k} : k \in \omega\}$ is unbounded in $I_{\rm wf}$ and hence in the other ideals as well.  It follows that $P$ is strongly unbounded.  This completes the proof.
\end{proof}

\section{Ideals below NWD}\label{belownwd}

This section explores the relationship between NWD and countably generated $F_\sigma$ ideals.  Specifically, all $F_\sigma$ ideals which are Tukey-below NWD must be countably generated (Theorem~\ref{fsigma}).

\begin{definition}
If $I$ is an ideal, a set $S \subseteq I$ is {\em $\sigma$-bounded} iff there is a countable set $C \subseteq I$ such every $X \in S$ is contained in some $Y \in C$.  If $I$ itself is $\sigma$-bounded, $I$ is called {\em countably generated}.
\end{definition}

\begin{theorem}\label{fsigma}
If $I$ is an $F_\sigma$ ideal and $I \leq_{\rm Tukey} {\rm NWD}$, then $I$ is countably generated.
\end{theorem}

\begin{lemma}\label{L1}
If $I$ is an ideal and $S = \bigcup_n Z_n$ is a subset of $I$ which is not $\sigma$-bounded, there is an $n$ such that $Z_n$ is also not $\sigma$-bounded.
\end{lemma}

\begin{proof}
Were this not the case, let $C_n \subseteq I$ witness that each $Z_n$ is $\sigma$-bounded.  The set $C = \bigcup_n C_n$ shows that $S$ is $\sigma$-bounded, a contradiction.
\end{proof}

%

In what follows, let $N(s)$ ($s \in 2^{<\omega}$) denote the basic clopen neighborhood
\[
N (s) = \{ y \in 2^\omega : s \prec y\}
\]
in the Cantor space.

\begin{lemma}\label{L3}
If $I$ is an $F_\sigma$ ideal and $Z_k \subseteq I$ ($k \in \omega$) are unbounded, there are finite sets $H_k \subseteq Z_k$ such that $\bigcup_k H_k$ is unbounded.
\end{lemma}

\begin{proof}
Fix an $F_\sigma$ ideal $I$ and let $F_0 , F_1 , \ldots \subseteq 2^\omega$ be such that $I = \bigcup_k F_k$ and each $F_k$ is topologically closed and closed under taking subsets.  This is possible since the downward closure of a compact set is still compact.

Suppose now that $Z_0 , Z_1 , \ldots$ are unbounded subsets of $I$.  Hence, for each $k \in \omega$, the set $\bigcup Z_k$ is in none of the $F_n$ and, in particular, $\bigcup Z_k \notin F_k$.  Thus, there are $p_k \in \omega$ such that
\[
 N \left( \left(\bigcup Z_k\right) \upharpoonright p_k \right) \cap F_k = \emptyset
\]
for each $k \in \omega$ since the $F_k$ are closed sets.  Let $n_k \in \omega$ and $x^k_0 , \ldots , x^k_{n_k} \in Z_k$ be such that 
\[
\left(x^k_0 \cup \ldots \cup x^k_{n_k}\right) \upharpoonright p_k = \left( \bigcup Z_k \right) \upharpoonright p_k.
\]
and hence
\[
 N \left( \left( x^k_0 \cup \ldots \cup x^k_{n_k} \right) \upharpoonright p_k \right) \cap F_k = \emptyset
\]
Let $H_k = \{ x^k_0 , \ldots , x^k_{n_k}\}$.  To see that $\bigcup_k H_k$ is unbounded, suppose that, on the contrary, 
\[
x = \bigcup_k \bigcup H_k \in I.
\]
Say $x \in F_i$ and hence $ N (x \upharpoonright p) \cap F_i \neq \emptyset$ for all $p \in \omega$.  Thus, since 
\[
x \supseteq x^i_0 \cup \ldots \cup x^i_{n_i}
\]
and $F_i$ is closed under taking subsets, it follows that
\[
 N \left( \left( x^i_0 \cup \ldots \cup x^i_{n_i} \right) \upharpoonright p_i \right) \cap F_i \neq \emptyset.
\]
Given the choice of $x^i_0 , \ldots , x^i_{n_i}$, this is a contradiction, completing the proof.
\end{proof}

It is now possible to prove the main result of this section.

\begin{proof}[Proof of Theorem~\ref{fsigma}]
Suppose that $I$ is an $F_\sigma$ ideal which is not countably generated and let $\pi : I \rightarrow {\rm NWD}$ be any map.  The objective is to show that $\pi$ is not Tukey, i.e., there is an unbounded set in $I$ whose image under $\pi$ is bounded in NWD.

Fix an enumeration $t_0 , t_1 , \ldots $ of $2^{<\omega}$.  The first step is to obtain inductively $s_k \in 2^{<\omega}$ and $Z_k \subseteq I$ such that
\begin{itemize}
	\item $Z_0 \supseteq Z_1 \supseteq \ldots$,
	\item $Z_k$ is not $\sigma$-bounded,
	\item $s_k \succeq t_k$ and,
	\item for each $x \in Z_k$, the intersection
	\[
	\pi (x) \cap \left( N( s_0 ) \cup \ldots \cup N( s_k ) \right)
	\]
	is empty.
\end{itemize}
To begin the induction, let
\[
Z_{0,s} = \{ x \in I : \pi (x) \cap N( s ) = \emptyset\}.
\]
for each $s \succeq t_0$.  Since the $\pi (x)$ (for $x \in I$) are all nowhere dense, the $Z_{0 , s}$ ($s \succeq t_0$) must cover $I$.  Since $I$ is not countably generated, it follows from Lemma \ref{L1} that there is some $s \succeq t_0$ such that $Z_{0 , s}$ is not $\sigma$-bounded.  Let $Z_0 = Z_{0 , s}$ and $s_0 = s$.  

To complete the induction, suppose that $Z_0 , \ldots , Z_k$ and $s_0 , \ldots , s_k$ are given satisfying the properties above.  To produce $Z_{k+1}$, let
\[
Z_{k+1 , s} = \{ x \in Z_k : \pi (x) \cap N( s ) = \emptyset \}
\]
for each $s \succeq t_{k+1}$.  The $Z_{k+1 , s}$ cover $Z_k$ and by Lemma \ref{L1} applied to $Z_k$, there is a $s \succeq t_{k+1}$ such that $Z_{k+1 , s}$ is not $\sigma$-bounded.  Let $Z_{k+1} = Z_{k+1 , s}$ and $s_{k+1} = s$.

This completes the construction of the $Z_k$.

Use Lemma \ref{L3} and the fact that non-$\sigma$-bounded sets are always unbounded to choose finite $H_k \subseteq Z_k$ such that such that $\bigcup_k H_k$ is unbounded in $I$.  The following claim now suffices to show that $\pi$ is not a Tukey map.

{\em Claim.}  The set $K = \bigcup \{ \pi (x) : x \in H_k \ \& \ k \in \omega\}$ is nowhere dense in $2^\omega$.

Indeed, fix a binary string $t \in 2^{<\omega}$, say $t = t_n$.  By choice of the $Z_k$ and since $H_k \subseteq Z_k$, it follows that $\pi (x) \cap N( s_n ) = \emptyset$ for each $x \in H_k$ with $k \geq n$.  In other words, 
\[
K \cap N( s_n ) = \left(\bigcup \{ \pi (X) : X \in H_k \ \& \ k < n\}\right) \cap N(s_n).
\]
As a finite union of nowhere dense sets, the set above is nowhere dense.  Hence, there is a string $s \succeq s_n \succeq t$ such that $K \cap N( s ) = \emptyset$.  As $t$ was arbitrary, it follows that $K$ is indeed nowhere dense.  This completes the claim and proof.
\end{proof}

\section{Ideals of well-founded subtrees}\label{wftrees}

This section turns to the weak Rudin-Keisler order and its structure on the ideals of well-founded sets of strings.  There are a number of natural questions which remain open.  These are discussed at the end of the section.  Recall from the introduction that 
\[
I_\alpha = \{ X \subseteq 2^{<\omega} : \mbox{there is an order-preserving } f : (X , \succ) \rightarrow (\alpha , <) \}.
\]
As mentioned above, $I_\alpha$ is an ideal iff $\alpha$ is additively closed.  It is a standard fact that all additively closed ordinals are powers of $\omega$.  With this in mind, the next theorem shows that $I_\omega$ is wRK-minimal among ideals of the form $I_\beta$ where $\beta$ is a power of a successor.

\begin{theorem}\label{omegaturkey}
If $\alpha$ is any countable ordinal, $I_\omega \leq_{\rm wRK} I_{\omega^{\alpha + 1}}$.
\end{theorem}

For the sake of the proof of Theorem~\ref{omegaturkey}, some additional notation and observations will be helpful.  
\begin{itemize}
	\item For $s \in 2^{<\omega}$ and a well-founded $T \subseteq 2^{<\omega}$, let 
	\[
	s{}^\smallfrown T = \{ s{}^\smallfrown t : t \in T\}.
	\]
	That is, $s{}^\smallfrown T$ is a copy of $T$ extending the string $s$.
	
	\item For well-founded $S , T \subseteq 2^{<\omega}$, let
	\[
	S {}^\smallfrown T = \{ s {}^\smallfrown t : \mbox{$s$ is a terminal node of $S$ and } t \in T  \}.
	\]
	In other words, $S {}^\smallfrown T$ consists of a copy of $T$ extending each terminal node of $S$.  If $T$ does not contain the empty string, $S \nsubseteq S {}^\smallfrown T$.
	
	\item Let $T^n$ denote the $n$-fold  ``sum'' $T {}^\smallfrown \ldots {}^\smallfrown T$.
	
	\item Observe that, if ${\rm rank} (S, \preceq) = \alpha$ and ${\rm rank} (T , \preceq) = \beta$, then 
	\[
	{\rm rank} (S {}^\smallfrown T , \preceq) = \beta
	\]
	and
	\[
	{\rm rank} (S \cup (S {}^\smallfrown T) , \preceq) = \beta + \alpha.
	\]
	
	
	
	\item Let $T_\alpha$ be a set of strings not containing the empty string and having ${\rm rank} (T_\alpha , \preceq) = \omega^\alpha$.  
	
	\item Observe that, for $n_0 < n_1 < \ldots n_k$, the rank of $\bigcup_{i \leq k} T_\alpha ^{n_i}$ is $\omega^\alpha \cdot k$.
	
	\item Let $\phi : 2^{<\omega} \rightarrow 2^{<\omega}$ be a function such that, if $s , t \in 2^{<\omega}$ and $s \neq t$, then $\phi (s) \perp \phi (t)$.
	
	\item Finally, let $\lambda_n$
\end{itemize}

\begin{proof}[Proof of Theorem ~\ref{omegaturkey}]
The first step is to define the desired wRK map $f$.  Using the notation above, given $s , t \in 2^{<\omega}$, let $f(s) = t$ iff there exists $t' \succeq t$ and $n \geq |t|$ such that
\[
s \in \phi (t') {}^\smallfrown \lambda_n {}^\smallfrown (T_\alpha)^{|t|}.
\]
Notice that, if $ t_1 \preceq \ldots \preceq t_k$ and $t \succeq t_k$, then $f^{-1} [ \{ t_1 , \ldots , t_k\}]$ contains a copy of
\[
\bigcup_{1 \leq i \leq k} (T_\alpha)^{|t_i|}
\]
extending each $\phi (t) {}^\smallfrown \lambda_n$ with $n \geq |t_k|$.  It follows by remarks before the proof that each $\phi (t)$ has rank $\omega^\alpha \cdot k$ in $f^{-1} [ \{ t_1 , \ldots , t_k\}]$.

Thus, if all $\prec$-chains of strings in a set $X \subseteq 2^{< \omega}$ have length at most $k \in \omega$, each $\phi (t)$ has rank at most $\omega^\alpha \cdot k$ in $f^{-1} [ X]$.  On the other hand, if $X \subseteq 2^{\omega}$ has $\prec$-chains of arbitrary length, there are $\phi (t) \in f^{-1} [X]$ in with rank $\omega^\alpha \cdot k$ (in $f^{-1} [X]$) for infinitely many $k$ -- namely those $t$ which are the terminal nodes of $\prec$-chains in $X$.  In particular, ${\rm rank} (f^{-1} [X] , \preceq)$ is at least $\omega^{\alpha + 1}$.
\end{proof}

It turns out that all ideals of the form $I_\alpha$ are wRK-incomparable with $I_{\rm wf}$, even though these ideals all have maximal Tukey type.  This is a consequence of part (c) of the following proposition and the remark following it.  

\begin{proposition}\label{nowrk}  \rule{0ex}{1ex}

{\em a)}  $\ell_1 \nleq_{\rm wRK} I_\omega$

{\em b)}  $\ell_1 \nleq_{\rm wRK} I_{\rm wf}$

{\em c)}  $I_\omega \nleq_{\rm wRK} I_{\rm wf}$
\end{proposition}

\begin{remark*}
Since $I_{\rm wf}$ is proper $\mathbf \Pi^1_1$ and $I_\omega$ is $F_\sigma$ it follows that $I_{\rm wf} \nleq_{\rm wRK} I_\omega$ since wRK-maps preserve definable complexity.  Also note that $I_\omega \nleq_{\rm wRK} \ell_1$ since the former is of maximal Tukey type and the latter is not.
\end{remark*}

\begin{proof}[Proof of Theorem \ref{nowrk}]
(a)  On the contrary, suppose that $f $ is a wRK-reduction of $\ell_1$ to $I_\omega$.  For each $n \in \omega$, the interval $I_n = [\frac{1}{2^n \cdot (n+1)} , \infty)$ is not in $\ell_1$.  Hence, $f^{-1} [I_n] \notin I_\omega$.  in light of this, choose $k_{n,0} , \ldots , k_{n,n} \in I_n$ such that 
\[
f^{-1} [ \{ k_{n,0} , \ldots , k_{n,n} ]
\]
has tree rank at least $n$.  Let $X = \{ k_{n,i} : i \leq n\}$.  It follows that $X \in \ell_1$, but $f^{-1} [ X] \notin I_\omega$.  This is a contradiction.

\vspace{1em}

\noindent (b)  Again suppose that $f$ is a wRK-reduction of $\ell_1$ to $I_{\rm wf}$.  Since $\omega \notin \ell_1$, there must be a infinite path $y \in 2^\omega$ in $\left[ f^{-1}[\omega] \right]$.  First note that $f$ is finite-to-one on 
\[
\{ t \in 2^{<\omega} : t \prec y\}.
\]
since each $\{ n \}$ is in $\ell_1$.  Use this fact inductively to pick $k_0 < k_1 < \ldots$ and $t_0 \prec t_1 \prec \ldots \prec y$ such that
\begin{itemize}
	\item $k_n \geq 2^n$ and
	\item $t_n \in f^{-1} [\{ k_n \}]$.
\end{itemize}
This yields a contradiction since $\{ k_n : n \in \omega \} \in \ell_1$, but $y \in \left[ f^{-1} [\{ k_n : n \in \omega\} \right]$, i.e., $f^{-1} [ \{ k_n : n \in \omega \}] \notin I_{\rm wf}$.

\vspace{1em}

\noindent (c) Suppose that the contrary is true.  Say $A \subseteq 2^{<\omega}$ is infinite and $f : A \rightarrow 2^{<\omega}$ is a wRK-reduction of $I_\omega$ to $I_{\rm wf}$.  The objective is to define a $\prec$-antichain $Z \subseteq 2^{<\omega}$ such that $f^{-1} [Z] \notin I_{\rm wf}$.  This will be accomplished by induction.  For $i , k \in \omega$, let
\[
s_{k,i} = 1^k {}^\smallfrown 0^{i+1}
\]
and 
\[
X = \{ s_{k,i} : i \leq k \in \omega\}.
\]
Since $X$ has infinite tree-rank, $X \notin I_\omega$ and hence $f^{-1} [X]$ contains an infinite path.  Say $y \in \left[ f^{-1} [X] \right]$.  To begin the induction, let $k_0 , i_0$ (with $i_0 \leq k_0$) be such that $f^{-1} \left[ \{ s_{k_0 , i_0} \} \right]$ contains a string $t_0 \prec y$.

Suppose now that an increasing sequence $k_0 < \ldots < k_n$, natural numbers $i_0 , \ldots , i_n$  and strings $t_0 \prec \ldots \prec t_n \prec y$ are given such that, for all $p \leq n$,
\[
t_p \in f^{-1} \left[ \{ s_{k_p , i_p} \} \right]
\]
Let
\[
Y_n = \{ s_{k,i} : i \leq k \leq k_n \}
\]
and note that $Y_n$ has rank $k_n$.  In particular, $Y_n \in I_\omega$.  Since $f$ is a wRK-reduction of $I_\omega$ to $I_{\rm wf}$, it follows that $f^{-1} [Y_n] \in I_{\rm wf}$ and hence contains only finitely many $t \prec y$.  Therefore, let $k_{n+1} > k_n$, $i_{n+1}$ and $t_{n+1} \in 2^{<\omega}$ be such that 
\begin{itemize}
	\item $t_n \prec t_{n+1} \prec y$ and
	\item $t_{n+1} \in f^{-1} \left[ \{ s_{k_{n+1} , i_{n+1}} \} \right]$.
\end{itemize}
Since $k_0 < k_1 < \ldots$, it follows that 
\[
Z = \{ s_{k_n , i_n} : n \in \omega\}
\]
is an antichain and hence in $I_\omega$.  On the other hand, $t_n \in f^{-1} [ Z ]$ for all $n$.  Hence, $y \in \left[ f^{-1} [Z] \right]$ and in particular $f^{-1} [Z] \notin I_{\rm wf}$.  This contradicts the assumption that $f$ is a wRK-reduction of $I_\omega$ to $I_{\rm wf}$.
\end{proof}

\subsection*{Some questions}

Some questions are left open by the results of this section.  The most obvious of these is the following:

\begin{question}
Is $I_\omega \leq_{\rm wRK} I_\alpha$ if $\alpha = \omega^\lambda$ for some limit ordinal $\lambda$?
\end{question}

\noindent The proof of Theorem~\ref{omegaturkey} is dependent on the fact that an ordinal of the form $\omega^{\alpha + 1}$ is the limit of ordinals of the form $\beta \cdot n$, for a fixed $\beta < \omega^{\alpha + 1}$.  This is of course not the case for ordinals $\omega^\lambda$ with $\lambda$ a limit.  

A more general question in the spirit of the previous one is

\begin{question}
Given arbitrary additively closed $\alpha  < \beta < \omega_1$, is $I_\alpha \leq_{\rm wRK} I_\beta$?
\end{question}

To motivate the next question, first observe that $\mathcal Z_0$ (the ideal of subsets of $\omega$ with asymptotic density zero) is not wRK-reducible to $I_\omega$ since the  latter is $F_\sigma$ and the former is $F_{\sigma \delta}$.  This leaves the following question.

\begin{question}
Is there some additively closed $\alpha$ such that $\mathcal Z_0 \leq_{\rm wRK} I_\alpha$?  For instance, is $\mathcal Z_0 \leq_{\rm wRK} I_{\omega^2}$?
\end{question}

\section{Locally compact Polishable ideals}\label{lcpol}

This section should be regarded as a ``warm-up'' for the main theorem Section~\ref{pideals}.  The proof of the following theorem gives an idea of the type of coding used in the proof Theorem~\ref{completeP}.

\begin{theorem}\label{completeLCP}
There is a locally compact Polishable ideal $I_*$ such that $J \leq_{\rm wRK} I_*$ for each locally compact Polishable ideal $J$.
\end{theorem}

In other words, there is a wRK-complete locally compact Polishable ideal.

Solecki \cite[Corollary 3.3]{solecki.ideals} provides a characterization of locally compact Polishable ideals which is crucial to the proof of Theorem~\ref{completeLCP}: an ideal $I \subseteq \omega$ is locally compact Polishable iff there is a set $A \subseteq \omega$ such that 
\[
I = \{ X \subseteq \omega : X \cap A \mbox{ is finite}\}.
\]
The ``if'' part of Solecki's result is the difficult part.  The ``only if'' part follows by observing that
\[
d(X,Y) = |(X \bigtriangleup Y) \cap A| + \sum_{n \in X \bigtriangleup Y} 2^{-n}
\]
is a complete metric on the ideal $\{ X \subseteq \omega : X \cap A \mbox{ is finite}\}$.

\begin{proof}[Proof of Theorem \ref{completeLCP}]
Instead of describing a wRK-complete locally compact Polishable ideal on $\omega$, the proof proceeds by defining such a wRK-complete ideal on an alternative countably infinite set, namely $2^{<\omega}$.  Specifically, let 
\[
B = \{ s {}^\smallfrown 1 : s \in 2^{<\omega}\}.
\]
and define
\[
I_* = \{ X \subseteq 2^{<\omega} : X \cap B \mbox{ is finite}\}.
\]
The ideal $I_*$ is locally compact and Polishable.  To see that $I_* $ is wRK-complete among such ideals, let $J$ be any locally compact Polishable ideal on $\omega$ and suppose $A \subseteq \omega$ is such that
\[
J = \{ X \subseteq \omega : X \cap A \mbox{ is finite}\}.
\]
Let $\alpha \in 2^\omega$ be the characteristic function of $A$ and define a map
\[
f: \{ \alpha \upharpoonright n : n \in \omega \} \rightarrow \omega
\]
by $f(\alpha \upharpoonright n) = n$.

{\em Claim.}  $f$ is a wRK-reduction of $J$ to $I_*$.

This follows from the observation that, for $X \subseteq \omega$,
\[
f^{-1} [X] = \{ \alpha \upharpoonright n : n \in X\}
\]
Therefore, $X \cap A$ is finite iff $f^{-1} [X] \cap \{ \alpha \upharpoonright n : \alpha (n) = 1\}$ is finite since $\alpha$ is the characteristic function of $A$.  In other words, $X \in J$ iff $f^{-1} [X] \cap B$ is finite, i.e., iff $f^{-1} [X] \in I_*$.
\end{proof}

\section{Analytic P-ideals}\label{pideals}

The next theorem establishes the existence of a wRK-complete analytic P-ideal.

\begin{theorem}\label{completeP}
There is an analytic P-ideal $I_{\rm max}$ such that $J \leq_{\rm wRK} I_{\rm max}$ for each analytic P-ideal $J$.
\end{theorem}

The proof of this theorem relies heavily on Solecki's characterization of analytic P-ideals in \cite{solecki.ideals}.  To utilize Solecki's result, it is necessary to first understand submeasures on $\omega$.

\begin{definition}
If $\mathcal A$ is a subalgebra of $\mathcal P (B)$ which contains the finite sets, a function $\phi : \mathcal A \rightarrow \mathbb R \cup \{ \infty \}$ is a {\em submeasure} iff
\begin{itemize}
	\item $\phi ( \emptyset ) = 0$,
	\item $\phi$ is {\em subadditive}, i.e., $\phi (X \cup Y) \leq \phi (X) + \phi (Y)$ and
	\item $\phi$ is {\em monotone}, i.e., $X \subseteq Y \implies \phi (X) \leq \phi (Y)$.
\end{itemize}
Such a map $\phi$ is {\em lower semicontinuous} (abbreviated {\em lsc})  iff,
\[
\lim_n \phi (X \cap n) = \phi (X)
\]
for each $X$ in the domain of $\phi$.
\end{definition}

\begin{remark*}
The definition above is not the most general definition of lower semicontinuity for functions on the Cantor space, but is equivalent to the general definition in the case of submeasures.
\end{remark*}

\begin{definition}
If $\phi$ is a lsc submeasure on $\mathcal P (B)$ for some countable set $B$, the {\em exhaustive ideal} of $\phi$ (written ${\rm Exh} (\phi)$) is the ideal consisting of those sets $X \subseteq B$ such that, for each $\varepsilon > 0$, there exists a finite set $F \subseteq B$ with $\phi ( X \setminus F ) < \varepsilon$.
\end{definition}

Given an lsc submeasure $\phi$, every ideal of the form ${\rm Exh} (\phi)$ is an analytic P-ideal.  In fact, ${\rm Exh} (\phi)$ is always $F_{\sigma \delta}$.  A powerful theorem of Solecki \cite[Theorem 3.1]{solecki.ideals} states that the converse is true as well: every analytic P-ideal is of the form ${\rm Exh} (\phi)$ for some lsc submeasure $\phi : \mathcal P (\omega) \rightarrow \mathbb R$.  In light of this, the proof of Theorem~\ref{completeP} proceeds by defining a ``universal'' lsc submeasure which encodes all possible lsc submeasures on $\omega$.

Before proving Theorem~\ref{completeP}, two lemmas are necessary.  The following lemma gives a way of extending a submeasure on finite sets to a lsc submeasure.  It is a standard result.

\begin{lemma}\label{p-ideals lemma 1}
If $\tilde \pi : [\omega]^{<\omega} \rightarrow \mathbb R \cup \{ \infty \}$ is a submeasure, the map $\pi : 2^\omega \rightarrow \mathbb R \cup \{ \infty \}$ defined by 
\[
\pi (X) = \sup \{ \tilde \pi (F) : F \subseteq X \mbox{ is finite} \} 
\]
is a lsc submeasure.
\end{lemma}

\begin{proof}
Let $\tilde \pi$ and $\pi$ be as in the statement of the lemma.

{\em Claim.}  $\pi$ is a submeasure.

To check monotonicity, suppose $X \subseteq Y \subseteq \omega$.  By definition,
\[
\pi (X) = \sup \{ \tilde \pi (F) : F \subseteq X \mbox{ is finite} \}  \leq \sup \{ \tilde \pi (F) : F \subseteq Y \mbox{ is finite} \}  = \pi (Y)
\]
To verify subadditivity, suppose that $X , Y \subseteq \omega$.
\begin{align*}
\pi (X \cup Y)
&= \sup \{ \tilde \pi (F) : F \subseteq X \cup Y \mbox{ is finite}\}\\
&\leq \sup \{ \tilde \pi (F \cap X ) + \tilde \pi (F \cap Y) : F \subseteq X \cup Y \mbox{ is finite} \}\\
&\leq \sup \{ \tilde \pi (F )  : F \subseteq X \mbox{ is finite} \} +  \sup \{ \tilde \pi (F ) : F \subseteq  Y \mbox{ is finite} \}\\
&= \pi (X) + \pi (Y)
\end{align*}

{\em Claim.}  $\pi$ is lsc.

The goal is to show that $\lim_n \pi (X \cap n) = \pi (X)$ for each $X \subseteq \omega$.  Fix $\varepsilon > 0$ and suppose $F \subseteq X$ with
\[
\pi (X) - \varepsilon < \tilde \pi (F) \leq \pi (X).
\]
Thus, for any $n > \max (F)$, 
\[
\pi (X) - \varepsilon < \tilde \pi (F) \leq \tilde \pi (X \cap n) = \pi (X \cap n) \leq \pi (X)
\]
Since $\varepsilon $ was arbitrary, this proves the claim.
\end{proof}

The next lemma is implicit the proof of Theorem 2.1 in Solecki \cite{solecki.ideals}.  

\begin{lemma}\label{p-ideals lemma 2}
Given any lsc submeasure $\psi : 2^\omega \rightarrow \mathbb R \cup \{ \infty \}$ there is a lsc submeasure $\pi : 2^\omega \rightarrow \mathbb R \cup \{ \infty \}$ such that 
\begin{itemize}
	\item ${\rm Exh} (\pi) = {\rm Exh} (\psi)$ and 
	\item $\pi (F) \in \mathbb Q$ for each $F \in [\omega]^{<\omega}$.
\end{itemize}
\end{lemma}

\begin{proof}
Let $\psi$ be any lsc submeasure on $2^\omega$.  Following the proof of Theorem 2.1 in Solecki \cite{solecki.ideals}, define
\[
\pi_1 (F) = \inf \{ 2^{-n} : \psi (F) \leq 2^{-n}\}
\]
for finite $F \subseteq \omega$.  Note that $\pi_1$ is $\mathbb Q$-valued and monotone, but may not be subadditive.  Also notice that, if $\pi_1 (F) = 2^{-n}$, then $2^{-n} \geq \psi (F) > 2^{-n-1}$, i.e., 
\[
\textstyle \pi_1 (F) \geq \psi (F) > \frac{1}{2} \pi_1 (F).
\]
Again for finite $F \subseteq \omega$, let
\[
\pi_2 (F) = \inf \left\{ \pi_1 (F_0) + \ldots + \pi_1 (F_k) : F_0 , \ldots , F_k \mbox{ are finite and cover } F\right\}.
\]
Since $\pi_1$ is monotone, it does not change the values of $\pi_2$ to assume that $F = F_0 \cup \ldots \cup F_k$ in the infimum above.  It is also safe to assume that each $F_i$ is nonempty.  Thus, the infimum above is over a finite set and, in particular, is always rational since $\pi_1$ takes only rational values.

{\em Claim.}  $\pi_2$ is a submeasure on $[\omega]^{<\omega}$.

Monotonicity follows from the fact that, if $F \subseteq G$ are finite sets, any cover of $G$ is also a cover for $F$.  Hence, the infimum which gives $\pi_2 (F)$ is over a larger set than the infimum which gives $\pi_2 (G)$ and so $\pi_2 (F) \leq \pi_2 (G)$.  

To check subadditivity, suppose that $F , G \in [\omega]^{<\omega}$ with  $F = F_0 \cup \ldots \cup F_k$
and $G = G_0 \cup \ldots \cup G_n$ such that
\[
\pi_2 (F) = \pi_1 (F_0) + \ldots + \pi_1 (F_k) \qquad \& \qquad \pi_2 (G) = \pi_1 (G_0) + \ldots + \pi_1 (G_n)
\]
Since
\[
F \cup G \subseteq F_0 \cup \ldots \cup F_k \cup G_0 \cup \ldots \cup G_n
\]
it follows that 
\[
\pi_2 (F \cup G) \leq \pi_1 (F_0) + \ldots + \pi_1 (F_k) + \pi_1 (G_0) + \ldots + \pi_1 (G_n) = \pi_2 (F) + \pi_2 (G).
\]
This proves that $\pi_2$ is subadditive and completes the claim.

\vspace{1em}

As in Solecki's argument, the following claim is the key to the proof.

{\em Claim.}  $\textstyle \pi_2 (F) \geq \psi (F) > \frac{1}{2} \pi_2 (F)$ for each finite set $F \subseteq \omega$.

To see this, suppose that $F_0 , \ldots , F_k$ are such that $F = F_0 \cup \ldots \cup F_k$ and $\pi_2 (F) = \pi_1 (F_0) + \ldots + \pi_1 (F_k)$.  Observe the following:
\begin{align*}
\pi_2 (F)
& = \pi_1 (F_0) + \ldots + \pi_1 (F_k) &\mbox{(by assumption)}\\
&\geq \psi (F_0) + \ldots + \psi (F_k) &\mbox{(by the definition of $\pi_1$)}\\
&\geq \psi (F) &\mbox{(by the subadditivity of $\psi$)}\\
&> \textstyle \frac{1}{2} \pi_1 (F) &\mbox{(again by the definition of $\pi_1$)}\\
&\geq \textstyle \frac{1}{2}\pi_2 (F) &\mbox{(since $\{ F \}$ covers $F$)}
\end{align*}
This establishes the claim.

\vspace{1em}

Now let
\[
\pi (X) = \sup \{ \pi_2 (F) : F \subseteq X \mbox{ is finite}\}.
\]
By Lemma \ref{p-ideals lemma 1}, $\pi$ is a lsc submeasure.  Also, $\pi$ agrees with $\pi_2$ on finite sets by the monotonicity of $\pi_2$.

{\em Claim.}  ${\rm Exh} (\pi) = {\rm Exh} (\psi)$

Suppose $X \in {\rm Exh} (\psi)$.  Fix $\varepsilon > 0$ and let $n \in \omega$ be such that $\psi ( X \setminus n) < \varepsilon / 2$.  Let $F \subseteq X \setminus n$ be finite.  By the previous claim, $\pi_2 (F) \leq 2 \psi (F) < \varepsilon$.  Since $F$ was arbitrary, this shows that $\pi (X \setminus n) \leq \varepsilon$.  As $\varepsilon$ was arbitrary, $X \in {\rm Exh} (\pi)$.

Reverse the roles of $\psi$ and $\pi$ and use the inequality $\psi (F) \leq \pi_2 (F)$ from the previous claim to show that ${\rm Exh} (\pi) \subseteq {\rm Exh} (\psi)$.  This completes the proof of the lemma.
\end{proof}

\begin{proof}[Proof of Theorem~\ref{completeP}]

The objective of the proof is to define an analytic P-ideal which is wRK-complete among analytic P-ideals.  In light of Solecki's characterization of such ideals, this wRK-complete ideal will be defined as the exhaustive ideal of a submeasure $\phi$.

In what follows, let $R_s \subseteq \omega^{<\omega}$ denote the interval consisting of all initial segments -- other than the empty string -- of a string $s \in \omega^{<\omega}$ with $|s| \geq 1$.  In other words, for $s$ of length $k$, 
\[
R_s = \{ \langle s(0) \rangle \, , \, \langle s(0) , s(1) \rangle \, , \, \ldots \, , \, \langle s(0) , s(1) , \ldots , s(k-1)\rangle \}.
\]
In particular, $R_s$ has cardinality $|s| - 1$.  In what follows, $R_s$ is identified with the interval $|s| = \{ 0 , \ldots , |s|-1\} \subseteq \omega$ via the map $(s \upharpoonright k) \mapsto (k-1)$.

For each $s \in \omega^{<\omega}$ with $|s| \geq 1$, choose a submeasure $\phi_s : \mathcal P (R_s) \rightarrow \mathbb Q$ such that the following hold.
\begin{itemize}
	\item For $s \in \omega^{<\omega}$ and $i \in \omega$, the submeasure $\phi_{s {}^\smallfrown i}$ agrees with $\phi_s$ on $\mathcal P (R_s)$;

	\item For each $q \in \mathbb Q$, there exists $i \in \omega$ such that $\phi_{\langle i \rangle} (\{ \langle i \rangle \}) = q$; 
	
	\item Supposing $s \in \omega^{<\omega}$ with $|s| = n$, if $\rho : \mathcal P ( n ) \rightarrow \mathbb Q$ is a submeasure such that $\rho \upharpoonright \mathcal P ( n-1 )$ agrees with $\phi_s$ via the above described identification between $n-1$ and $R_s$, there exists $i \in \omega$ such that $\rho$ agrees with $\phi_{s {}^\smallfrown i}$ via the identification between $n$ and $R_{s {}^\smallfrown i}$.
	
\end{itemize}
In other words, the third condition above implies that any $\mathbb Q$-valued submeasure on a finite set $A$ is ``coded'' by some $\phi_s$ with $|s| = |A| + 1$.  Also, by the first property of the $\phi_s$,
\begin{equation}\label{eq1}\tag{$\dagger$}
F \subseteq R_s \cap R_t \implies \phi_s (F) = \phi_t (F).
\end{equation}
for $s,t \in \omega^{<\omega}$.  The next step is to define a submeasure $\tilde \phi : [\omega^{<\omega}]^{<\omega} \rightarrow \mathbb Q$ which combines all of the $\phi_s$.  Given a finite set $F \subseteq \omega^{<\omega}$, let
\[
\tilde \phi (F) = \sup \{ \phi_s (F \cap R_s) : s \in \omega^{<\omega}\}.
\]
It follows from \ref{eq1} that this supremum need only be taken over those $s \in \omega^{<\omega}$ with $s \preceq t$ for some $t \in F$.  In particular, $\tilde \phi (F)$ is always well-defined and rational.  Moreover, for each $F \in [\omega^{<\omega}]^{<\omega}$, there exists $s \in \omega^{<\omega}$ such that $\tilde \phi (F) = \phi_s (F \cap R_s)$.

{\em Claim.}  The map $\tilde \phi : [\omega^{<\omega}]^{<\omega} \rightarrow \mathbb Q$ is a submeasure.

To verify monotonicity, suppose that $F \subseteq G \subseteq \omega^{<\omega}$ are finite sets.  Let $s \in \omega^{<\omega}$ be such that $\tilde \phi (F) = \phi_s (F \cap R_s)$.  It follows that
\[
\tilde \phi (F) = \phi_s (F \cap R_s) \leq \phi_s (G \cap R_s) \leq \tilde \phi (G).
\]
The first inequality above derives from the fact that $\phi_s$ is a submeasure.

To establish the subadditivity of $\tilde \phi$, fix finite sets $F , G \subseteq \omega^{<\omega}$ and let $s \in \omega^{<\omega}$ be such that $\tilde \phi (F \cup G) = \phi_s ((F \cup G) \cap R_s)$.  Thus,
\begin{align*}
\tilde \phi (F \cup G)
&= \phi_s ((F \cup G) \cap R_s)\\
&\leq \phi_s (F \cap R_s) + \phi_s (G \cap R_s)\\
&\leq \tilde \phi (F) + \tilde \phi (G)
\end{align*}
This completes the proof of the claim.

\vspace{1em}

Now let $\phi_{\rm max} : \mathcal P (\omega^{<\omega}) \rightarrow \mathbb R$ be given by
\[
\phi_{\rm max} (X) = \sup \{ \tilde \phi (F) : F \subseteq X \mbox{ is finite}\}.
\]
By Lemma \ref{p-ideals lemma 1}, $\phi_{\rm max}$ is itself a lsc submeasure.  The key property of ${\rm Exh} (\phi_{\rm max})$ is that, for $X  \subseteq \omega^{<\omega}$ and $\alpha \in \omega^\omega$,
\begin{equation}\label{eq2}\tag{$\ddagger$}
X \subseteq \bigcup_n R_{\alpha \upharpoonright n} \implies \phi_{\rm max} (X) = \lim_n \phi_{\alpha \upharpoonright n} (X \cap R_{\alpha \upharpoonright n}).
\end{equation}
This follows from the observation \ref{eq1} above.  The next claim will complete the proof.

{\em Claim.}  If $I$ is any P-ideal on $\omega$, then $I \leq_{\rm wRK} {\rm Exh} (\phi_{\rm max})$.

Given a P-ideal $I$ on $\omega$, it follows from Solecki \cite[Theorem 3.1]{solecki.ideals} that there is a lsc submeasure $\pi$ such that $I = {\rm Exh} (\pi)$.  By Lemma~\ref{p-ideals lemma 2}, it is no loss of generality to assume that $\pi (F) \in \mathbb Q$ for each finite set $F \subseteq \omega$.  By repeated application of the third property of the $\phi_s$, there is an $\alpha \in \omega^\omega$ such that
\[
\pi (F) = \phi_{\alpha \upharpoonright n} (\{ \alpha \upharpoonright (k+1) : k \in F\})
\]
for each finite $F \subseteq \omega$ and $n > \max(F) + 1$.  Define $f : \bigcup_n R_{\alpha \upharpoonright n} \rightarrow \omega$ by 
\[
f(\alpha \upharpoonright (n+1)) = n.
\]
To see that $f$ is a wRK-reduction of $I = {\rm Exh} (\pi)$ to ${\rm Exh} (\phi_{\rm max})$, first observe that, for each $X \subseteq \omega$, 
\[
f^{-1} [X] = \{ \alpha \upharpoonright (n+1) : n \in X\}.
\]
It now follows from the choice of $\alpha$ and $\ddagger$ that
\[
\phi_{\rm max} (f^{-1} [X]) = \pi (X)
\]
for each $X \subseteq \omega$.  Thus,
\[
X \in {\rm Exh} (\pi) \iff f^{-1} [X] \in {\rm Exh} (\phi_{\rm max})
\]
for $X \subseteq \omega$.  In other words, $f$ is a wRK-reduction of ${\rm Exh} (\pi)$ to ${\rm Exh} (\phi_{\rm max})$.  This verifies the claim and shows that $I_{\rm max} = {\rm Exh} (\phi_{\rm max})$ is a wRK-complete analytic P-ideal, completing the proof.
\end{proof}

\end{document}